\documentclass[12pt]{article}

\usepackage{amsfonts,amsmath,amssymb}
\usepackage{slashed,setspace}
\usepackage[all]{xy}
\usepackage{fullpage}
\usepackage{enumerate}
\usepackage[active]{srcltx}

\newtheorem{theorem}{Theorem}[section]
\newtheorem{definition}[theorem]{Definition}
\newtheorem{lemma}[theorem]{Lemma}
\newtheorem{remark}[theorem]{Remark}
\newtheorem{corollary}[theorem]{Corollary}
\newtheorem{proposition}[theorem]{Proposition}

\def\leftnote#1{\vadjust{\setbox1=\vtop{\hsize 20mm\parindent=0pt\bf\baselineskip=9pt\rightskip=4mm plus 4mm#1}\hbox{\kern-2cm\smash{\box1}}}}

\newenvironment{proof}[1][Proof]{\par\addvspace{2mm}\noindent\textbf{#1.} }{\ \rule{0.5em}{0.5em}\par\vspace{4mm}}

\newcommand{\Q}{\mathbb{Q}}
\newcommand{\Z}{\mathbb{Z}}
\newcommand{\F}{\mathbb{F}}
\newcommand{\PP}{\mathbb{P}}
\newcommand{\NN}{\mathbb{N}}
\newcommand\e{\hbox{\rm e}}
\newcommand{\R}{\mathcal{R}}
\newcommand{\A}{{\mathcal A}}

\newcommand{\bo}{{\mathfrak O}}
\newcommand{\D}{{\mathfrak D}}

\newcommand{\bp}{{\mathfrak P}}

\newcommand{\wb}{\overline}
\newcommand{\wt}{\widetilde}

\newcommand{\dd}{\partial}

\DeclareMathOperator{\Gal}{Gal}

\begin{document}
 \bibliographystyle{plain}  
 \title{Exponential power series, Galois module structure and differential modules}
 \author{Erik Jarl Pickett and St\'ephane Vinatier}
 \maketitle
 
\begin{abstract}
We use new over-convergent $p$-adic exponential power series, inspired
by work of Pulita \cite{Pulita}, to build self-dual normal basis
generators for the square root of the inverse different of certain
abelian weakly ramified extensions of an unramified extension $K$ of
$\Q_p$. 
%42
These extensions, whose set we denote by $\mathcal M$, are the degree 
$p$ subextensions over $K$ of $M_{p,2}$, the maximal
abelian totally, wildly and weakly ramified extension of $K$, whose
norm group contains $p$. 
%75
Our construction follows Pickett's \cite{Pickett}, who 
dealt with the same set $\mathcal M$ of extensions of $K$, but does
not depend on the choice of a basis of the residue field $k$ of $K$. 
%108%
Instead it furnishes a one-to-one correspondence, commuting with the
action of the Galois group of $K/\Q_p$, from the projective
space of $k$ onto $\mathcal M$. We describe very precisely the norm
group of the extensions in $\mathcal M$. When $K\not=\Q_p$, their
compositum $M_{p,2}$ yields an interesting example of non abelian weakly
ramified extension of $\Q_p$, with Galois group isomorphic to a wreath
product.  
%169%
Finally we show that, with a slight modification, our over-convergent
exponential power series endow certain differential modules with a
Frobenius structure, generalising a result of Pulita. Unfortunately,
they then lose the property we need to build self-dual normal basis
generators, hence the desirable link between Galois module structure
and differential modules is not yet obtained.
%225%
\end{abstract}

%\keywords{Exponential power series; Galois module structure;
%  Square-root of the inverse different; Weakly ramified extensions;
%  Differential module; Frobenius structure.} 

%\ccode{AMS Subject Classification: 11S31, 11S20, 11S23, 11R33, 12H25, 13N10}

%11S31 Class field th. + p-adic formal groups
%11S20 p-adic Galois th
%11S23 p-adic integral representations
%11R33 integral reps; Gal mod structure of ring of integers
%12H25 p-adic differential eqns
%13N10 Rings of differential operators and their modules

%%%%%%%%%%%%%%%%%%%%%% 
%%%%%%%%%%%%%%%%%%%%%% 
%%%%%%%%%%%%%%%%%%%%%% 

\section*{Introduction}

Exponential power series over $p$-adic rings are very useful objects
in a number of different fields. They have recently been used in the
description of rank one $p$-adic differential equations over the Robba
ring of a $p$-adic field \cite{Pulita} and to obtain results
concerning integral Galois module structure in wildly ramified
extensions of both local and global fields
\cite{Pickett,PickettVinatier}. In this paper we generalise and modify
some previous constructions of exponential power series and explore
their applications. 

Let $E/F$ be a finite odd degree Galois extension of number fields,
with Galois group $G$ and rings of integers $\bo_E$ and $\bo_F$. From
Hilbert's formula for the valuation of the different $\D_{E/F}$, see
\cite[IV \S2 Prop.4]{serre}, we know that $\D_{E/F}$ will have an even
valuation at every prime ideal of $\bo_E$ and we can thus define the
square-root of the inverse different $\A_{E/F}$ to be the unique
fractional $\bo_E$-ideal such that  
$$\A_{E/F}^2=\D_{E/F}^{-1}\enspace.$$
In \cite{erez2}, Erez proved that $\A_{E/F}$ is locally free over
$\bo_F[G]$ if and only if $E/F$ is at most weakly ramified,
\textit{i.e.}, when the second ramification groups are trivial at
every prime. This gives $\A_{E/F}$ the uncommon property that it is
locally free in certain wildly ramified extensions and as such, the
question of describing whether it is free over the group ring $\Z[G]$,
as Taylor famously achieved for the ring of integers of a tamely
ramified extension \cite{tay1}, raises new difficulties. Despite
partial results, see
\cite{erez2,PickettVinatier,Vinatier_jnumb,Vinatier-Surla,Vinatier-3},
this question remains open. 

We now let $p$ be a rational prime and let $\gamma$ be a root of the
polynomial $X^{p-1}+p$ in a fixed algebraic closure $\bar{\Q}_p$ of
the field of $p$-adic numbers $\Q_p$. Dwork's $p$-adic exponential
power series is defined as  
$$E_{\gamma}(X)=\exp(\gamma X-\gamma X^p)\enspace.$$
This power series was originally introduced by Dwork
in his study of the zeta function of hypersurfaces \cite{Dwork}. The convergence
properties of this power series can be used to endow a certain
$p$-adic differential module with a so-called Frobenius structure, see
Section \ref{Exp_section}, and such modules have become central
objects in the subjects of $p$-adic differential equations and
$p$-adic arithmetic geometry.

In \cite{Pickett} Pickett demonstrated how special values of Dwork's
power series can be used to construct arithmetic Galois module
generators in certain wildly ramified extensions of local fields
contained in Lubin-Tate division fields. In \cite{PickettVinatier}
Pickett and Vinatier use these constructions to make further progress
with the conjecture that $\A_{E/F}$ is free over $\Z[G]$  when $E/F$
is weakly ramified and of odd degree.  

This new approach marks the first progress to this goal for quite some
time; we now ask whether further progress can be made using other
techniques arising in the study of $p$-adic differential equations and
whether any insight can be gained in $p$-adic analysis with this new
perspective. We note that with the use of Fr\"ohlich's Hom-description
of the locally free class group, see \cite{Frohlich-Alg_numb}, much of
the work required for results in the global theory can be carried out
at a local level. We focus here on the local behaviour with the view
that a complete understanding of this will lead to results at a global
level. 

The main inspiration for this paper has come from recent work of
Pulita \cite{Pulita}, where he generalises Dwork's power series to a
class of exponentials, each with coefficients in a Lubin-Tate
extension of $\Q_p$: Let $f(X)\in\Z_p[X]$ be some Lubin-Tate
polynomial with respect to the uniformising parameter $p$,
\textit{i.e.},
$$f(X)\equiv X^p\mod p\Z_p[X] \text{\ \ \ and\ \ \ } f(X)\equiv pX\mod
X^2\Z_p[X]\enspace.$$  
Let $\{\omega_i\}_{i\ge0}$ be a \textit{coherent set of roots}
associated to $f(X)$, namely a sequence of elements of $\bar\Q_p$ such
that $f(\omega_i)=\omega_{i-1}$ and $\omega_0=0\ne\omega_1$. Then
define the $n$th Pulita exponential as 
$$E_n(X)=\exp\left(\sum_{i=0}^{n-1}\frac{\omega_{n-i}(X^{p^i}-X^{p^{i+1}})}{p^{i}}\right)\enspace.$$ 
These power series endow certain $p$-adic differential modules
with Frobenius structures and Pulita uses these modules to categorise
all rank one solvable $p$-adic differential equations. As Pulita
observes, restricting to the uniformising parameter $p$ means that
only Lubin-Tate division fields arising from a formal group isomorphic
to the multiplicative formal group are considered, see
\cite[\S3]{serre-lubintate} for an overview of the main theorems in
Lubin-Tate theory.

This paper is divided into five sections. In Section
\ref{sec:lubin-tate}, we set up the general framework. Assume $p$ is
an odd rational prime and $K$ is an unramified extension of $\Q_p$
with residue field $k$ of cardinality $q=p^d$. A finite
extension of $K$ is said to be \textit{weakly ramified} if it is
Galois and has a trivial second ramification group (in the lower
numbering). We first recall some properties of Galois groups of
Lubin-Tate division fields over $K$. Then, given a uniformising
parameter $\pi$ of $K$, we identify the maximal abelian totally,
wildly and weakly ramified extension of $K$, such that $\pi$ is a norm
from this extension, as the largest $p$-extension $M_{\pi,2}$
contained in the second Lubin-Tate division field $K_{\pi,2}$ relative
to $\pi$.

In Section \ref{Gal_mod_section} we describe how values of
modifications of Pulita's exponentials can be used to construct Galois 
module generators in the degree $p$ subextensions of $M_{p,2}/K$. 
Let $\mu_{q-1}=\{\mu\in\bar{\Q}_p:\mu^{q-1}=1\}$, then for
each $u\in\mu_{q-1}$, let $f_u(X)=X^p+upX$ and let $\omega_{u}$ be a
root of $f_u(X)-\gamma$ (recall that $\gamma$ is a root of
$f_1(X)/X$). We prove the over-convergence (see Notations and
Conventions for a definition) of the power series 
$$\mathcal{E}_{u,2}(X)=\exp\left(\omega_uX-u\omega_uX^p+\frac{\gamma
  X^p-\gamma X^{p^2}}{p}\right)\enspace.$$ 
We then show that if $v\in\mu_{q-1}$ and $u=v^{1-p}$, then
$\mathcal{E}_{u,2}(v)$ generates a Kummer extension $L_u$ over
$K(\gamma)=K(\zeta_p)$ (where $\zeta_p$ is any primitive $p$th root of 
unity), such that $L_u/K$ is cyclic of degree
$p(p-1)$. We show that its only subextension $M_u$ of degree $p$ over
$K$ is contained in $M_{p,2}$, hence weakly ramified, and that
$$\alpha_u=\frac{\sum_{s\in S}\mathcal{E}_{u,2}(v)^s}{p}$$
is a self-dual integral normal basis generator for $\A_{M_u/K}$, the
square-root of the inverse different of $M_u/K$, where
$S=\mu_{p-1}\cup\{0\}$. These basis generators are more natural than
the ones in \cite{Pickett} as they do not rely on a fixed basis of $k$
over $\F_p$.  
 
In Section \ref{class_field_section} we explicitly describe the norm
group of $M_u/K$, namely 
$$N(M_u/K)=\langle\,p\,\rangle\times\mu_{q-1}\times\exp(pv^{-p}Z)\enspace,$$ 
where $Z=\{x\in\bo_K:Tr_{K/\Q_p}(x)\in p\Z_p\}$ ($\bo_K$ is the
valuation ring of $K$). It follows that the map $v\mapsto M_u$ (recall
that $u=v^{1-p}$) is,
after natural identifications, a one-to-one correspondence between the
projective space $\PP(k)$ of the residue field $k$ and the set
$\mathcal M$ of subextensions of $M_{p,2}/K$ of degree $p$.
 
In Section \ref{non-ab_weak_section} we study $M_{p,2}$ as an absolute
extension of $\Q_p$. This extension is easily checked to be Galois;
when $[K:\Q_p]>1$, it yields an interesting example of a non-abelian
weakly ramified extension. We give a complete description of its
Galois group, showing that it is isomorphic to the regular wreath
product of a cyclic group of $p$ elements with a cyclic group of $d$
elements, $C_p  \wr C_d$. We also study the action of $\Gal(K/\Q_p)$
on some subextensions of $K_{p,2}/K$ and, as a corollary, we get that
the one-to-one correspondence of Section \ref{class_field_section}
commutes with the action of $\Gal(K/\Q_p)$. 
  
Our study of exponential power series from the view point of Galois
module structure has lead to the proof of the over-convergence of
power series that generalise the original constructions of
Pulita. Namely, in Section \ref{Exp_section} we prove the
over-convergence of the power series 
$$E_{u,n}(X)=\exp\left(\sum_{i=0}^{n-1}\frac{\omega_{u,n-i}(X^{p^i}-uX^{p^{i+1}})}{p^i}\right)$$
for all $u\in\Z_p^{\times}$ and $n\in\mathbb{N}$, where
$\{\omega_{u,i}\}_{i\ge0}$ is a coherent set of roots associated to a
Lubin-Tate polynomial with respect to the uniformising parameter
$up$. This removes Pulita's restriction that the Lubin-Tate formal
group be isomorphic to the multiplicative formal group. We then show
that these exponentials again furnish certain $p$-adic differential
modules with Frobenius structures.  

We remark that as differential modules with Frobenius structure and
modules with integral Galois structure are such important objects in
different subject areas, it is very desirable to obtain a direct link
between them. Unfortunately, despite using similar power series to
study these two structures, we can not provide a satisfactory link
with the results in this paper. We hope that future work will explore
this potential connection further. 

%%%%%%%%%%%%%%%%%%%%%%%%%%%%%%%%%%%%%%%%%%%%%%%%%%%%%%%
%%%%%%%%%%%%%%%%%%%%%%%%%%%%%%%%%%%%%%%%%%%%%%%%%%%%%%%
%%%%%%%%%%%%%%%%%%%%%%%%%%%%%%%%%%%%%%%%%%%%%%%%%%%%%%%

\section*{Notation and conventions}
 
We begin by fixing some notation and conventions which we will use
throughout this paper. 
\begin{itemize}
\item The set $\mathbb{N}$ is equal to the set of positive integers
  $\{1,2,3,\ldots\}$. For $n\in\mathbb{N}$ we denote the cyclic group
  of order $n$ by $C_n$.
\item Let $p$ denote an odd prime and let $\bar{\Q}_p$ be a fixed
  algebraic closure of $\Q_p$. For any $n\in\mathbb{N}$ we denote by
  $\mu_n$ the group of $n$th roots of unity contained in
  $\bar{\Q}_p^{\times}$.  
\item We let $|\cdot|_p$ be the normalised absolute value on
  $\bar{\Q}_p$ such that $|p|_p=p^{-1}$ and denote by $\mathbb{C}_p$
  the completion of $\bar{\Q}_p$ with respect to $|\cdot|_p$. We say
  that a power series with coefficients in some algebraic extension of
  $\Q_p$ is \textit{over-convergent} if it converges with respect to
  $|\cdot|_p$ on some disc $\{x\in\bar{\Q}_p:|x|_p< 1+\epsilon\}$ for
  some $\epsilon>0$. We note that an over-convergent function is
  convergent on the unit disc $\{x\in\bar{\Q}_p:|x|_p\le 1\}$. 
  \item For any extension $L/\Q_p$ considered we will always assume $L$
   is contained in $\bar{\Q}_p$ and we will denote by $\bo_L$, $\bp_L$
   and $k_L$ its valuation ring, valuation ideal and residue field
   respectively; we identify $k_{\Q_p}=\Z_p/p\Z_p$ with the field of
   $p$ elements, $\F_p$. If $\pi$ is a uniformising parameter of
   $\bo_L$, we denote by $L_{\pi,n}$ the $n$th Lubin-Tate division
   field with respect to $\pi$, see \cite[\S3.6]{serre-lubintate}, and
   we set $L_\pi=\bigcup\limits_{n\in\NN}L_{\pi,n}$.
\item Throughout we let $K/\Q_p$ be an unramified extension of degree
  $d\in\mathbb N$, we let $k_K=k$ and $q=p^d=|k|$. We know that
  $k^{\times}\cong\mu_{q-1}\subset K$ and using the Teichm\"uller
  lifting we will sometimes think of $k$ as lying in $\bo_K$ by
  identifying it with $\mu_{q-1}\cup\{0\}$. We let $\Sigma$ be the
  Galois group of $K/\Q_p$; $\Sigma$ is cyclic and generated by the
  unique lifting $\sigma$ of the Frobenius automorphism of the residue
  field extension. 
\item We denote by $K^{ab}$ the maximal abelian extension of $K$
  contained in $\bar{\Q}_p$ and let $\theta_K:K^{\times}\rightarrow
  \Gal(K^{ab}/K)$ be the Artin reciprocity map. For any extension
  $L/K$ with $L\subseteq K^{ab}$ we let $\theta_{L/K}$ be the
  composition of the Artin map with restriction of automorphisms to
  $L$; we denote by $N_{L/K}$ the norm map from $L$ to $K$ and by
  $N(L/K)$ the norm group $N(L/K)=N_{L/K}(L^\times)$.  
\end{itemize}

%%%%%%%%%%%%%%%%%%%%%%%%%%%%%%%%%%%%%%%%%%%%%%%%%%%%%%%
%%%%%%%%%%%%%%%%%%%%%%%%%%%%%%%%%%%%%%%%%%%%%%%%%%%%%%%
%%%%%%%%%%%%%%%%%%%%%%%%%%%%%%%%%%%%%%%%%%%%%%%%%%%%%%%

\section{Galois groups in Lubin-Tate towers}\label{sec:lubin-tate}

Let $\pi$ denote a uniformising parameter of $K$. Recall
that for $n\in{\mathbb N}$, the $n$th Lubin-Tate division field $K_{\pi,n}$ is
an abelian totally ramified extension of $K$, of norm group
$$N(K_{\pi,n}/K)=\langle\pi\rangle\times(1+\bp_K^n)$$ and
Galois group $\Gal(K_{\pi,n}/K)\cong
K^\times/N(K_{\pi,n}/K)\cong\mu_{q-1}\times(1+\bp_K)/(1+\bp_K^n)$. If
further $m\le n$ in $\mathbb N$, then $K\subset K_{\pi,m}\subseteq
K_{\pi,n}$ and 
\begin{equation}\label{galois-subgroup}
\Gal(K_{\pi,n}/K_{\pi,m})\cong1+\bp_K^m/1+\bp_K^{n}\enspace.
\end{equation}
See \cite{Iwasawa} for details, in particular Proposition 7.3 and
Lemma 7.4 (note that $K_{\pi,n}$ is denoted there by $K_\pi^{n-1}$).  
\begin{proposition}\label{GaloisGroups}
Let $n,m\in\mathbb{N}$ with $m\le n$, then
$$\Gal(K_{\pi,n}/K_{\pi,m})\cong\bo_K/p^{n-m}\bo_K\cong \bigoplus_{i=1}^d\big(\Z/{p^{n-m}}\Z\big)\enspace.$$
Further, any subextension of $K_\pi/K_{\pi,1}$ with Galois group of
exponent $p^{m-1}$ is contained in $K_{\pi,m}$.
\end{proposition}

\begin{proof}
Since $K/\Q_p$ is unramified and $p\ge 3$, we have
$\bp_K\subseteq\{x\in\mathbb C_p:|x|_p<p^{-\frac{1}{p-1}}\}$. Using
the Proposition in \cite[IV.1]{Koblitz} and the fact that $K$ is
complete, the exponential power series gives a group isomorphism from
the additive group $\bp_K$ to the multiplicative group $1+\bp_K$, with
inverse map the logarithmic power series. For any $i\in\mathbb{N}$,
this isomorphism yields 
$$\exp(\bp_K^i)=\exp(p^{i-1}\bp_K)=\exp(\bp_K)^{p^{i-1}}\cong(1+\bp_K)^{p^{i-1}}\enspace.$$ 
By \cite[XIV Proposition 9]{serre} we know that
$(1+\bp_K)^{p^{i-1}}\cong 1+\bp_K^i$. Combining these results we have
$\bp_K^i\cong1+\bp_K^i$ and thus, using (\ref{galois-subgroup}):
$$\Gal(K_{\pi,n}/K_{\pi,m})\cong
\bp_K^m/\bp_K^{n}=p^m\bo_K/p^{n}\bo_K\cong\bo_K/p^{n-m}\bo_K\enspace.$$ 
The first result now follows using a $\Z_p$-basis of $\bo_K$. Note
that $K_{\pi,n}/K_{\pi,1}$ has Galois group of exponent $p^{n-1}$. 

Suppose $L$ is a subextension of $K_\pi/K_{\pi,1}$ with Galois group
of exponent $p^{m-1}$, then $L\subseteq K_{\pi,n}$ for some
integer $n$, with $n\ge m$ by the previous result, namely
$K_{\pi,m}\subseteq K_{\pi,n}$. Any subgroup $H$ of
$G=\Gal(K_{\pi,n}/K_{\pi,1})$ such that $G/H$ is of exponent $p^{m-1}$
contains $G^{p^{m-1}}=\{g^{p^{m-1}}:g\in G\}$, which is of index
$p^{(m-1)d}=[K_{\pi,m}:K_{\pi,1}]$ in $G$, so that: 
$$\Gal(K_{\pi,n}/K_{\pi,m})=\Gal(K_{\pi,n}/K_{\pi,1})^{p^{m-1}}\enspace,$$
and $\Gal(K_{\pi,n}/K_{\pi,m})\subseteq\Gal(K_{\pi,n}/L)$,
hence $L\subseteq K_{\pi,m}$ as required.
\end{proof}
It follows from the proof that
$(1+\bp_K)/(1+\bp_K^n)\cong(C_{p^{n-1}})^d$ has order 
$q^{n-1}$, thus $\Gal(K_{\pi,n}/K)$ is the direct product of
$\mu_{q-1}$, which is cyclic of order prime to $p$, by an abelian
$p$-group.
\begin{definition} 
We define $M_{\pi,n}$ to be the unique subextension of $K_{\pi,n}/K$
such that $[M_{\pi,n}:K]=q^{n-1}$.
\end{definition}
Note that $K_{\pi,n}=M_{\pi,n}K_{\pi,1}$ and
$\Gal(M_{\pi,n}/K)\cong\Gal(K_{\pi,n}/K_{\pi,1})$ is of exponent
$p^{n-1}$. The subextensions of $M_{\pi,n}/K$ are characterised as
follows.
\begin{corollary}\label{coro:exponent}
Let $M$ be a finite abelian extension of $K$ with $\pi\in N(M/K)$,
then $M\subseteq M_{\pi,n}$ if and only if $\Gal(M/K)$ is of
exponent dividing $p^{n-1}$. 
\end{corollary}
\begin{proof}
Note that since $\pi\in N(M/K)$, $M\subset K_{\pi}$ by Lubin-Tate
theory. If $M\subseteq M_{\pi,n}$, the result is clear by Galois
theory. Suppose that $\Gal(M/K)$ is of exponent dividing $p^{n-1}$, so
that $M/K$ is a $p$-extension, then $M$ is linearly disjoint with
$K_{\pi,1}$ and the compositum $M'=MK_{\pi,1}$ has Galois group over
$K$ isomorphic to $\Gal(M/K)$, thus $M\subset M'\subseteq K_{\pi,n}$
by Proposition \ref{GaloisGroups}. The result follows since $M/K$ is a
$p$-extension. 
\end{proof}

Note also that $N(M_{\pi,n}/K)$ is a subgroup of
$K^\times$ containing $N(K_{\pi,n}/K)$ as a subgroup of index $q-1$;
it follows that 
$$N(M_{\pi,n}/K)=\langle\pi\rangle\times\mu_{q-1}\times(1+\bp_K^n)\enspace.$$
If $n\ge 2$, then $\bo_K^\times\not\subset N(M_{\pi,n}/K)$ and
$c(M_{\pi,n}/K)$, defined as the smallest integer $m$ such that
$1+\bp_K^{m}\subset N(M_{\pi,n}/K)$, equals $n$. In the following, we
are mostly interested in the case $n=2$. We characterise
$M_{\pi,2}$ as the maximal wildly (\textit{i.e.}, non tamely) and
weakly ramified abelian extension of $K$ in $K_\pi$. 
\begin{theorem}\label{weakly-ramified}
Let $M$ be a finite abelian extension of $K$ with $M\not=K$ and
$\pi\in N(M/K)$, then $M/K$ is wildly and weakly ramified if and only
if $M\subseteq M_{\pi,2}$. In particular, $M_{\pi,2}/K$ is weakly
ramified. 
\end{theorem}
\begin{proof}
We first show that $M_{\pi,2}/K$ is weakly ramified. Let $g_i$ denote 
the order of the $i$th ramification subgroup of $\Gal(M_{\pi,2}/K)$
(in the lower numbering). One has
$c(M_{\pi,2}/K)=2=\frac{g_0+g_1}{g_0}$, so the first assertion follows
from \cite[Coro. to Lemma 7.14]{Iwasawa}.  

Consider $M$ as in the Proposition, and note that $M/K$ is totally
ramified since $\pi\in N(M/K)$. Suppose $M\subseteq M_{\pi,2}$, then
$M/K$ is a $p$-extension, hence is wildly ramified; since
$M_{\pi,2}/K$ is weakly ramified, it follows from Herbrand's theorem
\cite[IV.3 Prop. 14]{serre} that the same holds for $M/K$,
as in the proof of \cite[Prop. 2.2]{Vinatier_jnumb}. Assume finally
that $M/K$ is wildly 
and weakly ramified, namely $h_1\not=h_2=1$, where $h_i$ denotes the
order of the $i$th ramification subgroup of $\Gal(M/K)$, hence
$h_0=h_1$ by \cite[IV.2 Coro. 2 to Prop. 9]{serre}. It follows that
$M/K$ is a $p$-extension and that $c(M/K)=2$, thus
$\langle\pi\rangle\times\mu_{q-1}\times(1+\bp_K^2)\subseteq
N(M/K)$. We get that
$$K^\times/\ \langle\pi\rangle\times\mu_{q-1}\times(1+\bp_K^2)
\cong(1+\bp_K)/(1+\bp_K^2)
\cong(C_p)^d$$
surjects onto $K^\times/N(M/K)\cong\Gal(M/K)$, which is thus of
exponent $p$, and we conclude using Corollary \ref{coro:exponent}. 
\end{proof}

%%%%%%%%%%%%%%%%%%%%%%%%%%%%%%%%%%%%%%%%%%%%%%%%%%%%%%%%%%%%
%%%%%%%%%%%%%%%%%%%%%%%%%%%%%%%%%%%%%%%%%%%%%%%%%%%%%%%%%%%%
%%%%%%%%%%%%%%%%%%%%%%%%%%%%%%%%%%%%%%%%%%%%%%%%%%%%%%%%%%%%

\section{Galois modules in Lubin-Tate extensions}\label{Gal_mod_section}

In \cite{Pickett} Pickett uses $p$th roots of special values of
Dwork's exponential power series to describe specific elements in
extensions $M/K$ such that: $M/K$ is weakly ramified, $[M:K]=p$ and
$p\in N(M/K)$. An element $\alpha_M\in M$ is constructed such that  
$$Tr_{M/K}(g(\alpha_M),h(\alpha_M))=\delta_{g,h}$$
for all $g,h\in\Gal(M/K)$, where $\delta$ is the Kronecker delta, and
such that $\alpha_M$ generates $\A_{M/K}$, the square-root of the
inverse different of $M/K$, as an integral Galois module: 
$$\A_{M/K}=\bo_K[\Gal(M/K)]\alpha_M\enspace.$$
Namely $\alpha_M$ is a \textit{self-dual normal basis generator} for 
$\A_{M/K}$.

In this section we explain how to modify Pulita's exponentials
$E_2(X)$, described in the Introduction, to construct alternative
self-dual normal basis generators for $\A_{M/K}$. This construction is
more canonical than the former as it does not rely on a specific
choice of basis of the residue field extension $k/\F_p$. 

%%%%%%%%%%%%%%%%%%%%%%%%%%%%%%%%%%%%%%%%%%%%%%%

\subsection{A new over-convergent exponential power series}

We recall that Dwork's power series is defined as
$E_{\gamma}(X)=\exp(\gamma X-\gamma X^p)$, where $\gamma$ is a root of
the polynomial $X^{p-1}+p$. It is over-convergent and converges to a
primitive $p$th root of unity when evaluated at any
$z\in\mu_{p-1}$, see \cite[Ch. 14 \S2-3]{LangII}. We let
$\zeta_p=E_{\gamma}(1)$ and note that $K(\gamma)/K$ is Kummer of
degree $p-1$, totally ramified with uniformising parameter
$\gamma$. We also have $K(\gamma)=K(\zeta_p)$ and we will now denote
this field by $K'$.  
\begin{definition}
For each $u\in\mu_{q-1}$, we define $f_u(X)\in\bo_K[X]$ as
$$f_u(X)=X^p+upX\enspace.$$ 
and let $\omega_{u}$ be a root of the Eisenstein polynomial
$f_u(X)-\gamma\in\bo_{K'}[X]$.
We then let $L_u=K'(\omega_{u})$ and define
$$\mathcal{E}_{u,2}(X)=\exp\left(\omega_uX-u\omega_uX^p+\frac{\gamma X^p-\gamma X^{p^2}}{p}\right)\enspace.$$
\end{definition}
Note that $\omega_u$ is also a root of the Eisenstein polynomial
$f_u(X)^{p-1}+p$ with coefficients in $\bo_{K}$, so $\omega_u$
generates $L_u$ over $K$ as well, namely $L_u=K(\omega_{u})$.

In the case $u=1$, we see that $\mathcal{E}_{1,2}(X)$ is equal to
Pulita's exponential $E_{2}(X)$ as described in the Introduction, with
$f(X)=f_1(X)$. Therefore, from \cite[Theorem 2.5]{Pulita}, we know 
$\mathcal{E}_{1,2}(X)$ is over-convergent. We deduce the following:
\begin{theorem}
For each $u\in\mu_{q-1}$, the power series $\mathcal{E}_{u,2}(X)$ is
over-convergent. 
\end{theorem}

\begin{proof}
For $u\in\mu_{q-1}$, we consider the power series 
$$A_u(X)
=\exp\left(\frac{f_1^2(\omega_{u}X)}{p^2}\right)$$ and note that
$A_u(X)=A_1(\frac{\omega_u}{\omega_1}X)$ is over-convergent if and
only if $A_1(X)$ is over-convergent. We note that $\gamma^{p-1}=-p$
and $\omega_{u}^p=\gamma-up\omega_{u}$; we make the following
derivation:   
\begin{eqnarray*}
{f_1^2(\omega_{u}X)}
&=&{(\omega_{u}^pX^p+p\omega_{u}X)^p+p(\omega_{u}^pX^p+p\omega_{u}X)}\\
&=&{((\gamma-up\omega_{u})X^p+p\omega_{u}X)^p+p((\gamma-up\omega_{u})X^p+p\omega_{u}X)}\\ 
&=&{(\gamma X^p+p\omega_{u}(X-uX^p))^p+p(\gamma X^p+p\omega_{u}(X-uX^p))}\\
&=&\gamma^pX^{p^2}+p^3B_u(X)+p\gamma X^p+p^2\omega_u(X-uX^p)
\end{eqnarray*}
for some $B_u(X)\in\bo_K[\omega_{u}][X]$. Therefore,
$$A_u(X)=\exp\Big(\omega_{u}X-\omega_{u}uX^p+\frac{\gamma X^p-\gamma
  X^{p^2}}{p}+pB_u(X)\Big)\enspace.$$ 
The power series $\exp(pB_u(X))$ is over-convergent and therefore
$\mathcal{E}_{u,2}(X)$ is over-convergent if and only if $A_{u}(X)$ is
over-convergent. From \cite[Theorem 2.5]{Pulita} we know that
$\mathcal{E}_{1,2}(X)$ is over-convergent, which implies $A_1(X)$ and
thus $A_u(X)$ are over-convergent. Therefore, $\mathcal{E}_{u,2}(X)$
is over-convergent.  
\end{proof}

%%%%%%%%%%%%%%%%%%%%%%%%%%%%%%%%%%%%%%%%%%%%%%%%%%%%%%%%%%%%

\subsection{Kummer and self-dual normal basis generators}

\begin{theorem}\label{KummerGenerator}
Let $v\in\mu_{q-1}$ and set $u=v^{1-p}$, then $\mathcal{E}_{u,2}(v)$ is a
Kummer generator of $L_u/K'$; further the extension $L_u/K$ is cyclic
of degree $p(p-1)$. 
\end{theorem}

\begin{proof} 
A straightforward calculation shows that:
$$\mathcal{E}_{u,2}(X)^p
=\exp(p\omega_uX)\exp(p\omega_uuX^p)^{-1}E_\gamma(X^p)\enspace,$$
where both $\exp(p\omega_uX)$ and $\exp(p\omega_uuX^p)$ are
over-convergent. Further, $v=uv^{p}$, and so
$\mathcal{E}_{u,2}(v)^p=E_{\gamma}(v^{p})$, which clearly belongs to
$K'=K(\gamma)$. 

By the properties of Dwork's power series one has
$E_{\gamma}(v^{p})\equiv 1+v^{p}\gamma\mod \gamma^2\bo_{K'}$, so
$E_{\gamma}(v^{p})\in 1+\bp_{K'}\setminus 1+\bp_{K'}^2\,$. Note that 
${K'}^\times=\langle\gamma\rangle\times\mu_{q-1}\times 1+\bp_{K'}$ and
that $(1+\bp_{K'})^p\subseteq 1+\bp_{K'}^p\,$, hence
$E_{\gamma}(v^{p})\notin(K'^{\times})^p$. It follows that
$K'\big(\mathcal{E}_{u,2}(v)\big)/K$ is Kummer of degree $p$. Since clearly
$K'\big(\mathcal{E}_{u,2}(v)\big)\subseteq K'(\omega_{u})=L_u$, a
comparison of the degrees then gives equality and it follows that
$L_u/K'$ is cyclic of degree $p$. 

We now show that $L_u/K$ is Galois. Let $\delta\in\Gal(K'/K)$, then
$\delta:\gamma\mapsto z\gamma$ for some $ z\in\mu_{p-1}$. Any
$w\in\bo_K^{\times}$ is fixed by $\delta$ and contained in the disc of
convergence of $E_{\gamma}(X)=\exp(\gamma X-\gamma
X^p)=\sum_{n=0}^{+\infty}e_nX^n$. One has: 
$$\sum_{n=0}^{+\infty}\delta(e_n)X^n=\exp(\delta(\gamma)X-\delta(\gamma)X^p)=\exp\big(\gamma(zX)-\gamma(zX)^p\big)=E_\gamma(zX)\enspace,$$
which implies
$$\textstyle
\delta(E_{\gamma}(w))=\delta\left(\sum\limits_{n=0}^{+\infty}e_nw^n\right)=\sum\limits_{n=0}^{+\infty}\delta(e_n)w^n=E_\gamma(zw)=E_\gamma(w)^z\enspace,
$$
using \cite[Lemma 4.2]{PickettVinatier}. Let $\wt{L}_u$ be the Galois 
closure of $L_u/K$ and let $\wt{\delta}\in\Gal(\wt{L}_u/K)$ be such
that $\wt{\delta}|_{K'}=\delta$. We observe that 
\begin{eqnarray*}
\wt{\delta}(\mathcal{E}_{u,2}(v))^p
=\wt{\delta}(\mathcal{E}_{u,2}(v)^p)
=\delta(E_{\gamma}(v^{p}))
=E_{\gamma}(v^{p})^z=\mathcal{E}_{u,2}(v)^{pz}\enspace.
\end{eqnarray*}
We must then have, for some $\zeta\in\mu_p$:
\begin{equation}\label{lifting}
\wt{\delta}(\mathcal{E}_{u,2}(v))=\zeta \mathcal{E}_{u,2}(v)^z\enspace,
\end{equation}
which means that all the Galois conjugates of
$\mathcal{E}_{u,2}(v)$ over $K$ are contained in $L_u$, and so
$L_u=K\big(\mathcal{E}_{u,2}(v)\big)$ is Galois over $K$. 

Each $\delta\in\Gal(K'/K)$ has $p$ liftings in $\Gal(L_u/K)$, given by 
the various choices of the $p$-th root of unity $\zeta$ in 
(\ref{lifting}). Suppose $\delta$ generates $\Gal(K'/K)$ and let
$\wt\delta$ now denote the lifting of $\delta$ corresponding to
$\zeta=1$, then $\wt\delta$ generates a subgroup of
$\Gal(L_u/K)$ of order $p-1$. Recall that $\zeta_p$ is a primitive
$p$-th root of unity, and note that
$\wt\delta(\zeta_p)=\zeta_p^z$ from \cite[Lemma
  10]{Pickett}. Let $\phi$ denote the generator
of $\Gal(L_u/K')$ defined by $\phi(\mathcal{E})=\zeta_p\mathcal{E}$,
where we let 
$\mathcal{E}=\mathcal{E}_{u,2}(v)$ for ease of notation. Then
$\phi\circ\wt\delta(\mathcal{E})
=\zeta_p^z\mathcal{E}^z
=\wt\delta\circ\phi(\mathcal{E})$,
namely $\Gal(L_u/K)$ is abelian, hence the result.
\end{proof}

The description of the extensions $L_u/K$ given in Theorem
\ref{KummerGenerator} is only available when $u\in\mu_{q-1}$ is a
$(p-1)$-th power, namely when $u$ belongs to the subgroup
$(\mu_{q-1})^{p-1}=\mu_{\frac{q-1}{p-1}}$ of $\mu_{q-1}$. If this is
not the case, we are not able to prove that $L_u/K'$ is Galois. 
However, we shall see below that the extensions described in the
theorem are the only ones of interest to us, see Remark \ref{remark}. 

From now on, we thus restrict ourselves to the units
$u\in(\mu_{q-1})^{p-1}$. It amounts to the same as restricting to the
units $v\in\mu_{q-1}/\mu_{p-1}$, since $v\mapsto v^{1-p}$ yields an
isomorphism between $\mu_{q-1}/\mu_{p-1}$ and $(\mu_{q-1})^{p-1}$. We
remark that $\sigma$ acts on both these groups and commutes with this
isomorphism. Let $\PP(k)=k^\times/\F_p^\times$ denote the projective
space of $k$ considered as an $\F_p$-vector space of dimension $d$; in
the following we identify $\PP(k)$ with $\mu_{q-1}/\mu_{p-1}$ through
the Teichm\"uller lifting. 

\begin{definition}\label{Mu_defn} 
For any $u\in(\mu_{q-1})^{p-1}$, we let $M_u$ be the unique
subextension of $L_u/K$ of degree $p$.
\end{definition}
\begin{proposition}\label{Mu_lemma}
Let $u\in(\mu_{q-1})^{p-1}$, then $M_u$ is the splitting field of
$\psi_u(X)=X(X+up)^{p-1}+p$ and
$$M_u=K(\omega_{u}^{p-1})\subseteq M_{p,2}\enspace.$$
\end{proposition}

\begin{proof}
By Theorem \ref{KummerGenerator}, $L_u=K(\omega_{u})$ is the splitting
field over $K$ of the polynomial  
$$f_1(f_u(X))=(X^p+upX)^p+p(X^p+upX)\enspace.$$
Therefore $L_u$ also splits the Eisenstein polynomial
$X^{p-1}(X^{p-1}+up)^{p-1}+p$ which is the minimum polynomial of
$\omega_{u}$ over $K$. This implies that the minimum polynomial of
$\omega_{u}^{p-1}$ over $K$ is $\psi_u(X)=X(X+up)^{p-1}+p$. 
The equality $M_u=K(\omega_{u}^{p-1})$ now follows as this polynomial
is of degree $p$ and $L_u/K$ has only one subextension of degree $p$. 
Further, $N_{M_u/K}(\omega_{u}^{p-1})=p$ and $\Gal(M_u/K)$ has
exponent $p$, therefore $M_u\subseteq M_{p,2}$ from Corollary
\ref{coro:exponent}. 
\end{proof}
It follows that $M_u/K$ is weakly ramified by Proposition
\ref{weakly-ramified}, hence its square-root of the inverse different
admits a self-dual normal basis generator, by \cite[Theorem 1]{erez2}
and \cite[Corollary 4.7]{Fainsilber+Morales}.

\begin{theorem}
Let $v\in\PP(k)$, set $u=v^{1-p}\in(\mu_{q-1})^{p-1}$, and let
$\A_{M_u/K}$ denote the unique fractional ideal of $\bo_{M_u}$ whose
square is equal to the inverse different of $M_{u}/K$. Set
$S=\mu_{p-1}\cup\{0\}$, then: 
$$\alpha_u=\frac{1+Tr_{L_u/M_u}(\mathcal{E}_{u,2}(v))}{p}=\frac{\sum_{s\in
    S}\mathcal{E}_{u,2}(v)^s}{p}$$  
is a self-dual integral normal basis generator for $\A_{M_u/K}$. 
\end{theorem}

\begin{proof}
As in the proof of the previous theorem we know that
$\Gal(L_u/M_u)=\langle\,\wt{\delta}\,\rangle$, where
$\wt{\delta}(\mathcal{E}_{u,2}(v))=\mathcal{E}_{u,2}(v)^z$ for some
primitive $(p-1)$th root of unity $z$. This proves the equivalence of
the presentations of $\alpha_u$. 

In order to use the methods in \cite{Pickett} we highlight the
following three properties of $\mathcal{E}_{u,2}(v)$: 
\begin{enumerate}
\item $\mathcal{E}_{u,2}(v)$ is a Kummer generator of $L_u/K'$. 
\item From \cite[Ch. 14 \S2]{LangII} we know that
  $E_{\gamma}(v^p)=1+v^p\gamma\mod\gamma^2$. Therefore,
  $E_{\gamma}(v^p)-1=\mathcal{E}_{u,2}(v)^{p}-1$ is a uniformising
  parameter of $K'$.
\item The involution element of $\Gal(L_u/M_u)$ maps
  $\mathcal{E}_{u,2}(v)$ to $\mathcal{E}_{u,2}(v)^{-1}$. 
\end{enumerate}
The proof that $\alpha_u$ generates a self-dual normal basis for
$\A_{M_u/K}$ now follows the exact method of proof of \cite[Theorem
  12]{Pickett}. 
\end{proof}

%%%%%%%%%%%%%%%%%%%%%%%%%%%%%%%%%%%%%%%%%%%%%%%%%%%%%%%
%%%%%%%%%%%%%%%%%%%%%%%%%%%%%%%%%%%%%%%%%%%%%%%%%%%%%%%
%%%%%%%%%%%%%%%%%%%%%%%%%%%%%%%%%%%%%%%%%%%%%%%%%%%%%%% 

\section{The degree $p$ subextensions of $M_{p,2}/K$}\label{class_field_section}
Recall that $M_{\pi,2}$ is the unique subfield of $K_{\pi,2}$ such 
that $[M_{\pi,2}:K]=q$, with norm group:
$$N(M_{\pi,2}/K)=\langle\,\pi\,\rangle\times\mu_{q-1}\times 1+\bp_K^2\enspace.$$ 
Specifically $M_{\pi,2}=M_{w\pi,2}$ for all $w\in\mu_{q-1}$. Since
$\Gal(M_{\pi,2}/K)\cong(C_p)^d$, we know that $M_{\pi,2}/K$ has 
$\frac{q-1}{p-1}=1+p+\cdots+p^{d-1}$ subextensions of degree
$p$. Recall that for each $u\in (\mu_{q-1})^{p-1}$ we have defined an 
extension $M_u/K$ of degree $p$, such that $M_u\subseteq M_{p,2}$, see
Proposition \ref{Mu_lemma}.

We now fix $u\in(\mu_{q-1})^{p-1}$ and calculate the norm group
$N(M_u/K)$ of $M_u/K$. As $M_u\subseteq M_{p,2}$, from standard local
class field theory we know that 
\begin{equation}\label{SubgroupOfNormGroup}
N(M_{p,2}/K)=\langle\,p\,\rangle\times\mu_{q-1}\times
1+\bp_K^2\ \subseteq\ N(M_u/K)\enspace. 
\end{equation}
Furthermore, the Artin map $\theta$ of $K$ yields the isomorphism:
\begin{equation}\label{IndexOfNormGroup}
K^{\times}/ N(M_u/K)\cong \Gal(M_u/K)\cong C_p\enspace.
\end{equation}
We introduce the $\Z_p$-module $Z$ defined by:
$$Z=\{x\in\bo_K:Tr_{K/\Q_p}(x)\in p\Z_p\}\enspace.$$
Note that $Z\varsubsetneq\bo_K$ since $K/\Q_p$ is unramified. Further 
$\bp_K\subseteq Z$, but $Z$ needs not be an ideal of $\bo_K$. In fact
the next result implies that $[\bo_K:Z]=p$, whereas
$[\bo_K:\bp_K]=p^d$, hence $\bp_K\varsubsetneq Z$.

We recall from above that as $p\ge 3$ and $K/\Q_p$ is unramified, for
any $n\in\mathbb{N}$ the power series $\exp$ and $\log$ give us group
isomorphisms between $1+\bp_K^n$ with the multiplicative structure and
$\bp_K^n$ with the additive structure.

\begin{theorem}\label{norm_thm}
Let $v\in\PP(k)$ then
$$N(M_{v^{1-p}}/K)=\langle\,p\,\rangle\times\mu_{q-1}\times\exp(pv^{-p}Z)\enspace.$$
As a consequence, the set $\{M_{v^{1-p}}:v\in\PP(k)\}$ gives a
complete set of intermediate fields of $M_{p,2}/K$ of degree $p$ over
$K$. Further let $v,v_1\not=v_2\in\PP(k)$, then $M_{v^{1-p}}\subset
M_{v_1^{1-p}}M_{v_2^{1-p}}$ if and only if $v$ belongs to the line of
$\PP(k)$ through $v_1$ and $v_2$. 
\end{theorem}

\begin{remark}\label{remark}
Before proving this theorem we address a question posed in
\cite[Remark 4.4]{PickettVinatier}. Namely, it was stated that it
would be desirable to describe a generating set for the degree $p$
extensions of $K'$ contained in $K_{p,2}$ that does not depend on the
choice of a basis of $k$ over $\Z_p/p\Z_p$. This theorem, together
with Theorem \ref{KummerGenerator}, shows that
$\{\mathcal{E}_{v^{1-p},2}(v):v\in\mathbb{P}(k)\}$ is such a set. 
\end{remark}

\begin{proof}
Set $u=v^{1-p}$, so $u$ is well defined in $(\mu_{q-1})^{p-1}$.
We already know by equality (\ref{SubgroupOfNormGroup}) that
$\langle\,p\,\rangle\times\mu_{q-1}\subseteq N(M_{u}/K)$.
We first show that $\exp(pv^{-p}Z)\subseteq N(M_{u}/K)$, then that 
$\langle\,p\,\rangle\times\mu_{q-1}\times\exp(pv^{-p}Z)$ is of index $p$
in $K^\times$. This is enough to obtain equality using
(\ref{IndexOfNormGroup}). 

We let $\eta\in\mu_{q-1}$ be such that $\eta\mod p$ generates a normal
basis of $k/\F_p$. It yields the following descriptions of $\bo_K$ and
$Z$. 
\begin{lemma}\label{decompositions}
One has the decompositions:
$\bo_K=\eta\Z_p\oplus\bigoplus_{i=1}^{d-1}(\eta^{p^i}-\eta^{p^{i+1}})\Z_p$ 
and $Z=p\eta\Z_p\oplus\bigoplus_{i=1}^{d-1}(\eta^{p^i}-\eta^{p^{i+1}})\Z_p$.
\end{lemma}
\begin{proof}
Since $K/\Q_p$ is unramified, the set of conjugates of $\eta$ under
$\Sigma$ forms a basis $(\eta^{p^i})_{0\le i\le d-1}$ for $\bo_K$ over
$\Z_p$. A determinant computation shows that $(\eta
,\eta^{p}-\eta^{p^{2}},\eta^{p^2}-\eta^{p^{3}},\ldots,\eta^{p^{d-1}}-\eta)$
is also a basis of $\bo_K$ over $\Z_p$, hence the decomposition of
$\bo_K$. Any element of
$\bigoplus_{i=1}^{d-1}(\eta^{p^i}-\eta^{p^{i+1}})\Z_p$ has zero trace
(this is in fact the kernel of the trace restricted to $\bo_K$), so
$Z\supseteq
p\eta\Z_p\oplus\bigoplus_{i=1}^{d-1}(\eta^{p^i}-\eta^{p^{i+1}})\Z_p$. Suppose 
$x=\sum_{i=0}^{d-1}x_i\eta^{p^i}\in Z$, where each $x_i\in\Z_p$, then
$Tr_{K/\Q_p}(x)=\big(\sum_{j=0}^{d-1}x_j\big)Tr_{K/\Q_p}(\eta)\in
p\Z_p$. 

Let $\wb y$ denote the reduction modulo $p\bo_K$ of $y\in\bo_K$,
then
$\overline{Tr_{K/\Q_p}(\eta)}=Tr_{k/\F_p}(\overline{\eta})\not=0$,
since $\wb\eta$ generates a normal basis, hence $p$ divides
$\sum_{j=0}^{d-1}x_j$. Therefore,
$$\textstyle
x=\Big(\sum\limits_{j=0}^{d-1}x_j\Big)\eta
+\sum\limits_{i=1}^{d-1}\Big(\sum\limits_{j=1}^ix_j\Big)(\eta^{p^i}-\eta^{p^{i+1}})
\displaystyle\ \in\ 
p\eta\Z_p\oplus\bigoplus_{i=1}^{d-1}(\eta^{p^i}-\eta^{p^{i+1}})\Z_p
$$
as required.
\end{proof}

We get 
$\exp(pv^{-p}Z)=\exp(p^2v^{-p}\eta\Z_p)\times\prod_{i=1}^{d-1}\exp\big(pv^{-p}(\eta^{p^i}-\eta^{p^{i+1}})\Z_p\big)$. Further 
$p^2v^{-p}\eta\Z_p\subseteq\bp_K^2$ hence
$\exp(p^2v^{-p}\eta\Z_p)\subseteq N(M_u/K)$ thanks to
(\ref{SubgroupOfNormGroup}). The inclusion $\exp(pv^{-p}Z)\subseteq
N(M_u/K)$ will thus be a consequence of the next result. 
\begin{lemma}
 Let $w\in k^\times$, then the multiplicative subgroup of $1+\bp_K$
 given by $\exp\big(pv^{-p}(w-w^p)\Z_p\big)$ is a subgroup of
 $N(M_{u}/K)$.  
\end{lemma}

\begin{proof}
Let $x\in\Z_p$, we then have 
$$\exp\big(pv^{-p}(w-w^p)x\big)
\equiv1+pv^{-p}(w-w^p){x}\mod\bp_K^2\enspace$$
and as subgroups of $1+\bp_K$,
$$\exp\big(pv^{-p}(w-w^p)\Z_p\big)(1+\bp_K^2)
=\big\langle 1+pv^{-p}(w-w^p)\big\rangle(1+\bp_K^2)\enspace.$$

We know that $1+\bp_K^2$ is a subgroup of
$N(M_{u}/K)$, therefore it is now sufficient to show
that some element congruent to $1+pv^{-p}(w-w^p)\mod\bp_K^2$ belongs to
$N(M_{u}/K)$. We see that 
$$\textstyle\psi_{u}\left(\frac{X-1}{wv^{-1}}\right)
=\frac{X-1}{wv^{-1}}\left(\frac{X-1}{wv^{-1}}+up\right)^{p-1}+p$$
and therefore the norm of $1+\omega_{u}^{p-1}wv^{-1}$ will equal
$$\textstyle\frac{-v}{w}\left(\frac{-v}{w}+up\right)^{p-1}+p
\equiv-(\frac{v}{w})^p+p\big(1-\frac{1}{w^{p-1}}\big)\mod\bp_K^2\enspace.$$
We then see that 
\begin{eqnarray*}
\textstyle N_{M_{u}/K}\big(-\frac{w}{v}(1+\omega_{u}^{p-1}wv^{-1})\big)
&=&\textstyle -(\frac{w}{v})^pN_{M_{u}/K}(1+\omega_{u}^{p-1}wv^{-1})\\
&\equiv& 1+pv^{-p}(w-w^p)\mod \bp_K^2\enspace,
\end{eqnarray*}
which proves the result.
\end{proof}

We now compute the index $I$ of
$\langle\,p\,\rangle\times\mu_{q-1}\times\exp(pv^{-p}Z)$ in
$K^\times=\langle\,p\,\rangle\times\mu_{q-1}\times 1+\bp_K$. One has: 
$$I=[\exp(\bp_K):\exp(pv^{-p}Z)]=[\bp_K:pv^{-p}Z]=[\bo_K:Z]=p\enspace,$$
using the decompositions in Lemma \ref{decompositions}. This completes
the norm group computation.

The extension $M_{v^{1-p}}$ of $K$ only depends on the class of
$v\in\mu_{q-1}$ modulo $\mu_{p-1}$, so the set
$\{M_{v^{1-p}}:v\in\mu_{q-1}\}$ contains at most $\frac{q-1}{p-1}$
extensions of $K$. This is the number of subextensions of $M_{p,2}/K$
of degree $p$, so we are left with showing that
$M_{v^{1-p}}\not=M_{w^{1-p}}$ whenever $v\not=w$ in $\PP(k)$. This is
equivalent to showing that the norm groups of these extensions are
different, namely that $v^{-p}Z\not=w^{-p}Z$ whenever $v\not=w$ in
$\PP(k)$.

For $x\in\bo_K$, we denote by $\wb x$ its reduction modulo $p\bo_K$.  
\begin{lemma}\label{KernelTrace}
Let $v\in\mu_{q-1}$, then 
$$v^{-p}Z=\{x\in\bo_K:Tr_{k/\F_p}(v^{p}\wb x)=0\}\enspace,$$
so the reduction of $v^{-p}Z$ modulo $p$ is the kernel of the linear
form $y\mapsto Tr_{k/\F_p}(v^{p}y)$ on the $\F_p$-vector space $k$. 
\end{lemma}
\begin{proof}
Let $x\in\bo_K$, then $x\in v^{-p}Z$ if and only if
$\overline{Tr_{K/\Q_p}(v^{p}x)}=Tr_{k/\F_p}(v^{p}\wb x)=0$, which
yields the result.
\end{proof}
The trace form $T(\nu,\nu')=Tr_{k/\F_p}(\nu\nu')$ is a non degenerate
bilinear pairing on $k\times k$, so it induces an isomorphism
$\nu\mapsto T(\nu,\,.\,)$ between $k$ and its dual. It follows
that, for $\nu,\nu'\in k^\times$, $T(\nu,\,.\,)$ and $T(\nu',\,.\,)$
have the same kernel if and only if $\nu/\nu'\in\F_p^\times$. Suppose
$v\not=w$ in $\PP(k)$ then the same holds for $v^{p}$ and $w^{p}$,
so $v^{-p}Z$ and $w^{-p}Z$ are different modulo $p$, hence are different.
It follows that $\{M_{v^{1-p}}:v\in\PP(k)\}$ is the complete set of
subextensions of degree $p$ of $M_{p,2}/K$.

By local class field theory one has $M_{v^{1-p}}\subset
M_{v_1^{1-p}}M_{v_2^{1-p}}$ if and only if $N(M_{v^{1-p}}/K)\supset
N(M_{v_1^{1-p}}/K)\cap N(M_{v_2^{1-p}}/K)$, namely if and only if $v^{-p}Z\supset
v_1^{-p}Z\cap v_2^{-p}Z$, where with a slight abuse of notation we now
see $v,v_1,v_2$ as elements of $\mu_{q-1}$ (since any of their
liftings can equally be chosen). By Lemma \ref{KernelTrace}, this is
also equivalent to
$$\forall w\in k\,,\quad Tr_{k/\F_p}(v_1^pw)=0=
Tr_{k/\F_p}(v_2^pw)\ \Rightarrow\ Tr_{k/\F_p}(v^pw)=0\enspace.$$
This implication is clearly true if there exist $a,b\in\F_p$ such that  
$v=av_1+bv_2$, namely if $v\pmod{\mu_{p-1}}$ belongs 
to the line of $\PP(k)$ through $v_1\pmod{\mu_{p-1}}$ and
$v_2\pmod{\mu_{p-1}}$. Since there are exactly $p+1$ points on this
line, which all correspond to a different subextension of 
$M_{v_1^{1-p}}M_{v_2^{1-p}}$ of degree $p$, and that there are exactly
$p+1$ such subextensions, we see that $M_{v^{1-p}}\subset
M_{v_1^{1-p}}M_{v_2^{1-p}}$ implies that $v$ belongs to the line of
$\PP(k)$ through $v_1$ and $v_2$. We have therefore proved the
declared equivalence.

This ends the proof of Theorem \ref{norm_thm}.%{\hfill\qed}
\end{proof}

\begin{corollary}
Let $v\in\PP(k)$ and set $u=v^{1-p}$, then
$$\Gal(M_u/K)
=\langle\,\theta_{M_u/K}(\exp(\eta v^{-p}p))\,\rangle
=\langle\,\theta_{M_u/K}(1+\eta v^{-p}p)\,\rangle
\enspace.$$  
\end{corollary}

\begin{proof}
Restricting the elements in $\Gal(K^{ab}/K)$ to their action on $M_u$
gives the short exact sequence: 
$$1\rightarrow N(M_u/K)\rightarrow
K^{\times}\stackrel{\theta_{M_u/K}}\longrightarrow
\Gal(M_u/K)\rightarrow 1\enspace.$$ 
Using the decompositions in Lemma \ref{decompositions} we see that 
$$K^{\times}/N(M_{u}/K)\cong\exp(\eta v^{-p}p\Z_p)/\exp(\eta v^{-p}p^2\Z_p)$$
which proves the first equality. The second equality comes from the 
fact that $\langle \exp(\eta v^{-p}p) \rangle(1+\bp_K^2)=\langle
1+\eta v^{-p}p \rangle(1+\bp_K^2)$ and $(1+\bp_K^2)\subseteq
N(M_{u}/K)$.
\end{proof}

%%%%%%%%%%%%%%%%%%%%%%%%%%%%%%%%%%%%%%%%%%%%%%%%%%%%%%%
%%%%%%%%%%%%%%%%%%%%%%%%%%%%%%%%%%%%%%%%%%%%%%%%%%%%%%%
%%%%%%%%%%%%%%%%%%%%%%%%%%%%%%%%%%%%%%%%%%%%%%%%%%%%%%%

\section{The absolute extension $M_{p,2}/\Q_p$}\label{non-ab_weak_section}

\begin{proposition}
Suppose $[K:\Q_p]>1$, then $M_{p,2}/\Q_p$ is a non-abelian weakly
ramified extension of local fields.
\end{proposition}
\begin{proof}
We first check that $M_{p,2}$ is Galois over $\Q_p$.
The polynomial $h(X)=X^q+pX$ is a Lubin-Tate polynomial over $\bo_K$
for uniformising parameter $p$. The extension $K_{p,2}/K$ is abelian,
of degree $q(q-1)$ and contains all the roots of the polynomial
$h(h(X))$. The polynomial  
$$H(X)=h(h(X))/h(X)=(X^q+pX)^{q-1}+p$$
is Eisenstein, and so irreducible, of degree $q(q-1)$ and therefore
$K_{p,2}/K$ is the splitting field of $H(X)$. With a straightforward
rearrangement we see that  
$$H(X)=X^{q-1}(X^{q-1}+p)+p=\bar{H}(X^{q-1})\enspace,$$ 
where $\bar{H}(X)=X(X+p)^{q-1}+p$ is also Eisenstein, hence
irreducible, and of degree $q$. As $M_{p,2}$ is the unique subfield of
$K_{p,2}$ of degree $q$ over $K$ it must be the splitting field of
$\bar H(X)$.

Both $H$ and $\bar H$ have coefficients invariant under the
action of $\Gal(K/\Q_p)$, therefore the extensions $K_{p,2}/\Q_p$ and
$M_{p,2}/\Q_p$ must be Galois: let $\omega$ denote a root of $H$, so
that $K_{p,2}=K(\omega)$, and let $\wt\sigma$ be a lifting of $\sigma$
in the Galois group of the Galois closure of $K_{p,2}$ over $\Q_p$, then
$H\big(\wt\sigma(\omega)\big)=\wt\sigma\big(H(\omega)\big)=0$. But
$K_{p,2}$ splits $H$, so $\wt\sigma(\omega)\in K_{p,2}$, which
therefore is normal over $\Q_p$. The same argument applies to
$M_{p,2}$. 

Since $K/\Q_p$ is unramified, it nows follows from Proposition
\ref{weakly-ramified} that $M_{p,2}/\Q_p$ is weakly ramified. If
$d=[K:\Q_p]>1$, $M_{p,2}/\Q_p$ can not be abelian since otherwise its
ramification index would equal $p$ by \cite[Th\'eor\`eme
  1.1]{Vinatier_jnumb}, and $M_{p,2}/K$ is totally ramified of degree
$p^d$.
\end{proof}
The fact that $M_{p,2}/\Q_p$ is Galois can also be shown by
considering that $N(M_{p,2}/K)$ is invariant under
$\Gal(K/\Q_p)$.

Let $\eta\in\mu_{q-1}$ be such that $\eta\mod p$ is a normal basis
generator of $k/\F_p$. Since $K/\Q_p$ is unramified, the set of 
conjugates of $\eta$ under $\Sigma$ forms a basis $(\eta^{p^i})_{0\le
  i\le d-1}$ for $\bo_K$ over $\Z_p$. This means that:
$$\bp_K=p\bo_K=p\bigoplus_{i=0}^{d-1}\eta^{p^i}\Z_p
=\bigoplus_{i=0}^{d-1}\eta^{p^i}p\Z_p\enspace.$$ 
\begin{lemma}
One has: $(1+\bp_K)/(1+\bp_K^2)=
\left(\prod\limits_{i=0}^{d-1}\langle1+\eta^{p^i}p\rangle\right)\Big/(1+\bp_K^2)
\enspace.$ 
\end{lemma}

\begin{proof}
The inclusion ``$\supseteq$'' is clear, so let us prove ``$\subseteq$''.
In view of the above decomposition of $\bp_K$, any $y\in 1+\bp_K$ can
be written as $y=1+p\sum_{i=0}^{d-1}y_i\eta^{p^i}$, with 
$y_i\in\Z_p$, so 
$y\equiv\prod_{i=0}^{d-1}(1+py_i\eta^{p^i})
\equiv\prod_{i=0}^{d-1}(1+p\eta^{p^i})^{y_i}\mod p^2$. Thus there
exists $a\in\bo_K$ such that 
$$\prod_{i=0}^{d-1}(1+p\eta^{p^i})^{y_i}=y+p^2a=y(1+p^2ay^{-1})\enspace,$$
and the result follows.
\end{proof}
%\comment{
%\textbf{Alternative proof.}
%Using the exponential isomorphism as in the previous section we then see 
%\begin{eqnarray*}
%1+\bp_K&=&\prod_{i=0}^{d-1}\exp(\eta^{p^i}p\Z_p)\\
%(1+\bp_K)/(1+\bp_K^2)&=&\prod_{i=0}^{d-1}\exp(\eta^{p^i}p\Z_p)/(1+\bp_K^2)\\
%&=&\prod_{i=0}^{d-1}\langle1+\eta^{p^i}p\rangle/(1+\bp_K^2)\enspace.
%\end{eqnarray*}
%\todo{prove the last equality!}
%}
Note that $\eta^{p^i}$ only depends on $i$ modulo $d$, so we let
$g_i=\theta_{M_{p,2}/K}(1+ \eta^{p^i} p)$ for $i\in\Z/d\Z$. As
the Galois group of $M_{p,2}/K$ is isomorphic, via the Artin map, to
$(1+\bp_K)/(1+\bp_K^2)$, we get: 
\begin{equation}
\Gal(M_{p,2}/K)=\langle\, g_{i}:
i\in\Z/d\Z\,\rangle=\prod_{i\in\Z/d\Z}\langle g_i\rangle\enspace.
\end{equation}
This gives an explicit description of the isomorphism
$\Gal(M_{p,2}/K)\cong (C_p)^d$.

We now describe explicitly the Galois group of $M_{p,2}/\Q_p$.
We let $\wt{\sigma}\in\Gal(K_{p,2}/\Q_p)$ be some lifting of $\sigma$,
\textit{i.e.}, $\wt{\sigma}|_{K}=\sigma$. Let again $\omega$ be a root
of $H(X)$, we then have $K_{p,2}=K(\omega)$ and
$M_{p,2}=K(\omega^{q-1})$, so $\wt\sigma$ is determined by its value
on $\omega$. We have seen that $\wt\sigma(\omega)$ is one of the
roots of $H$. It follows that taking the value at $\omega$
yields a one-to-one correspondence between the liftings of $\sigma$
and the roots of $H$. Consequently, we can pick $\wt{\sigma}$ such
that $\wt{\sigma}(\omega)=\omega$.

\begin{theorem}
Let $\wt\sigma\in\Gal(K_{p,2}/\Q_p)$ be such that
$\wt{\sigma}(\omega)=\omega$ and $\wt{\sigma}|_{K}=\sigma$, and
consider its restriction, denoted the same, to $M_{p,2}$. With the
definitions above, 
$$\Gal(M_{p,2}/\Q_p)
=\langle\,\wt{\sigma},\ g_{i}: i\in\Z/d\Z\,\rangle\enspace.$$
Further, $\wt{\sigma} \circ g_{i}=g_{{i+1}}\circ\wt{\sigma}$ and 
$\Gal(M_{p,2}/\Q_p)$ is isomorphic to the regular wreath product
$C_p\wr C_d$. 
\end{theorem}
For a full definition of wreath products see, for example,
\cite[Ch.7]{Rotman_groups}. 
\begin{proof}
We first let $\wt{g}_i=\theta_{K_{p,2}/K}(1+\eta^{p^i}p)$ and note that 
$g_i=\wt{g}_i|_{M_{p,2}}$. We know that $1+\bp_K^2\subseteq N(K_{p,2}/K)$, 
and so $\theta_{K_{p,2}/K}(1+\eta^{p^i}p)=\theta_{K_{p,2}/K}((1+\eta^{p^i}p)a)$ 
for any $a\in 1+\bp_K$. Clearly
$(1+\eta^{p^i}p)^{-1}\equiv(1-\eta^{p^i}p)\mod (1+\bp_K^{2})$, so from
standard local class field theory we have 
$$\wt{g}_{i}(\omega)
=\theta_{K_{p,2}/K}(1+{\eta^{p^i}}p)(\omega)
=[1-{\eta^{p^i}}p](\omega)
\enspace,$$ 
where $[1-{\eta^{p^i}}p](X)$ is the Lubin-Tate power series
associated to the polynomial $h(X)$ and the unit
$(1-\eta^{p^i}p)\in\bo_K^{\times}$ (see \cite[\S3.3
  Prop.2]{serre-lubintate} and \cite[\S3.4 Theorem
  3(c)]{serre-lubintate}). Considering the power series expansion of 
$[1-{\eta^{p^i}}p](X)$ we see that 
\begin{eqnarray*}
\wt{\sigma}\circ \wt g_{i}(\omega)
&=&\wt{\sigma}\big([1-{\eta^{p^i}}p](\omega)\big)\\ 
&=&[1-{\sigma}({\eta^{p^i}})p]\big(\wt{\sigma}(\omega)\big)\\
&=&[1-{\eta^{p^{i+1}}}p](\omega)\\
&=&\wt g_{i+1}(\omega)\\
&=&\wt g_{i+1}\circ\wt{\sigma}(\omega)\enspace.
\end{eqnarray*}
We must therefore also have $\wt{\sigma}\circ
g_i(\omega^{q-1})=g_{i+1}\circ\wt{\sigma}(\omega^{q-1})$. 

As $g_i$ acts trivially on $\eta$ for all $i$, we clearly have 
$g_i\circ\wt{\sigma}(\eta)=\wt{\sigma}\circ g_{i+1}(\eta)$. 
It now follows that $g_i\circ\wt{\sigma}=\wt{\sigma}\circ g_{i+1}$, as 
$M_{p,2}=\Q_p(\omega^{q-1},\eta)$. 

Finally we obtain the isomorphism $\Gal(M_{p,2}/\Q_p)\cong C_p\wr C_d$
by comparing cardinalities. 
\end{proof}

We finish our study of $K_{p,2}/\Q_p$ by describing the action of the
element $\wt{\sigma}\in\Gal(K_{p,2}/\Q_p)$ on the subfields $L_u$ and
$M_u$. 

\begin{proposition}\label{prop:galois-action} 
With $v\in\PP(k)$ and $u=v^{1-p}$, we have
$$\wt{\sigma}(L_u)=L_{\sigma(u)} \text{\ \ \  and\ \ \ }
\wt{\sigma}(M_u)=M_{\sigma(u)}\enspace.$$ 
Moreover, $M_u/\Q_p$ is Galois if and only if $\sigma(u)=u^p=u$.
\end{proposition}

\begin{proof}
We recall that $L_u=K(\omega_u)$ is the splitting field of the
minimum polynomial of $\omega_u$ over $K$, $(X^p+upX)^{p-1}+p$. We
have  
$$(\omega_u^p+up\omega_u)^{p-1}+p=0\enspace,$$ 
and so 
$$(\wt{\sigma}(\omega_u)^p+\sigma(u)p\wt{\sigma}(\omega_u))^{p-1}+p=0\enspace.$$
This means that $\wt{\sigma}(\omega_u)$ is a root of
$(X^p+\sigma(u)pX)^{p-1}+p$ and as $L_{\sigma(u)}$ is the splitting
field of this polynomial, we must have
$L_{\sigma(u)}=K(\wt{\sigma}(\omega_u))=\wt{\sigma}(L_u)$. The fact
that $\wt{\sigma}(M_u)=M_{\sigma(u)}$ follows from the same reasoning
using $\omega_u^{p-1}$ and its minimum polynomial in the place of
$\omega_u$. 

As $M_{p,2}/\Q_p$ is Galois and $M_u\subseteq M_{p,2}$, all embeddings
of $M_u$ into $\bar{\Q}_p$ will be obtained by restricting the action
of elements of $\Gal(M_{p,2}/\Q_p)$ to $M_u$. We have seen that
$\Gal(M_{p,2}/\Q_p)=\langle\,\wt{\sigma},\ g_{i}: i\in\Z/d\Z\,\rangle$.

For all $i$, $g_i\in\Gal(M_{p,2}/K)$, and so $g_i|_{M_u}$ is an
element of $\Gal(M_u/K)\subseteq\Gal(M_{u}/\Q_p)$, by Galois
theory. Therefore, $M_u/\Q_p$ is Galois if and only if 
$M_u=\wt{\sigma}(M_u)$. We have just shown that
$\wt{\sigma}(M_u)=M_{\sigma(u)}$ and from Theorem \ref{norm_thm} we
know that $M_u=M_{\sigma(u)}$ if and only if $u=\sigma(u)$ which
completes the proof.  
\end{proof}
\begin{remark}
We know by Theorem \ref{norm_thm} that the map $v\mapsto
v^{1-p}$ is a one-to-one correspondence from $\PP(k)$ onto the set
$\mathcal M$ of degree $p$ extensions of $K$ contained in
$K_{p,2}$. We have noticed earlier that $\sigma$ acts on $\PP(k)$; let
it act on $\mathcal M$ through $\wt\sigma$, then it follows from
Proposition \ref{prop:galois-action} that the above bijection commutes
with the action of $\sigma$.
\end{remark}

It seems natural to ask whether
$\wt{\sigma}(\alpha_u)=\alpha_{\sigma(u)}$, with suitable choices of 
$\omega_u$ and possibly a different choice of $\wt{\sigma}$. However,
we are not able to answer this question at present.

%%%%%%%%%%%%%%%%%%%%%%%%%%%%%%%%%%%%%%%%%%%%%%%%%%%%%%%
%%%%%%%%%%%%%%%%%%%%%%%%%%%%%%%%%%%%%%%%%%%%%%%%%%%%%%%
%%%%%%%%%%%%%%%%%%%%%%%%%%%%%%%%%%%%%%%%%%%%%%%%%%%%%%%

\section{Exponential power series and differential Frobenius structures}\label{Exp_section}

We now discuss how Pulita's method in \cite{Pulita} can be modified to
give further results concerning differential modules with Frobenius
structure. For each Lubin-Tate formal group over $\Q_p$, Pulita
constructs a set of exponential power series and proves that all
exponentials in this set are over-convergent if and only if the formal
group is isomorphic to the multiplicative formal group (see
\cite[Theorem 2.1]{Pulita}), \textit{i.e.}, the formal group is
associated to the uniformising parameter $p$. 

In this section we describe Pulita's main results used in the
construction of his exponential power series, omitting some technical
details which are described fully in \cite{Pulita}. We describe how a
different choice of Witt vector in this construction allows us to
relax the condition that the Lubin-Tate formal group must be
isomorphic to the multiplicative formal group in order for each of the
associated exponentials to be over-convergent. We then describe how
these exponentials can be used to describe differential modules with a
Frobenius structure. 

These results were found during our investigation of exponential power
series from the point of view of Galois module structure, and we hoped
that they would describe some kind of link between these two subject
areas. However, we have not achieved this elusive goal. 
Indeed, the only power series from this section that can be used
directly in the constructions in Section \ref{Gal_mod_section} is
$E_{1,2}(X)$ --- see Theorem \ref{thm:power-series}. It has the very
special property that $E_{1,2}(u)$ is a primitive $p^2$th root of
unity for $u\in\mu_{p-1}$, and so $E_{1,2}(X)$ only gives us
information about cyclotomic extensions. In order to obtain Galois
module generators in extensions not obtained from cyclotomy, we had to
modify the power series $E_{1,2}(X)$ as in Section
\ref{Gal_mod_section}. These modified power series then lose the
properties that we use to endow a differential module with a Frobenius
structure in this section. Using the methods here, apart from in the
cyclotomic case, the power series considered can not give information
about both differential Frobenius structure and integral Galois module
structure.  

%%%%%%%%%%%%%%%%%%%%%%%%%%%%%%%%%%%%%%%%%%%%%%%%%%%%%%%%%%%%%%%%%%%%

\subsection{Pulita exponentials arising from different Witt vectors} 
 
\begin{definition}
The $n$th Witt polynomial is defined as 
 $$\mathcal{W}_{(n)}(X_0,\ldots,X_n)
 =X_0^{p^n}+pX_1^{p^{n-1}}+\ldots+p^nX_n\enspace.$$ 
For any ring $R$, the ring of Witt vectors over $R$, denoted $W(R)$,
is equal as a set to $R^{\mathbb{N}\cup\{0\}}$ and endowed with the
ring structure such that the map  
$$\begin{array}{cccc}
\mathcal{W}:&W(R)&\longrightarrow&R^{\mathbb{N}\cup\{0\}}\\ 
&(r_0,r_1,\ldots)&\longmapsto&(\mathcal{W}_{(0)}(r_0),\mathcal{W}_{(1)}(r_0,r_1),\ldots)
\end{array}$$
is a ring homomorphism. For $r=(r_0,r_1,\ldots)\in{W}(R)$, we define  
$$r^{(n)}=\mathcal{W}_{(n)}(r_0,r_1,\ldots,r_n)
\text{\ \ \ and\ \ \ }\mathcal{W}(r)=<r^{(0)},r^{(1)},\ldots>$$
to be the $n$th ghost component and the ghost vector of $r$
respectively. 
\end{definition}
See, for example, \cite{Fesenko-Vostokov} for full details.

For $u\in\Z_p^{\times}$ we let $f_u(X)\in\Z_p[X]$ be such that 
$$f_u(X)\equiv X^p \mod p\Z_p[X] 
\text{\ \ \ \ and\ \ \ \ } 
f_u(X)\equiv upX\mod X^2\Z_p[X]$$
and let $\{\omega_{u,i}\}_{i\ge0}$ be a coherent set of roots
associated to $f_u(X)$. Note that $\Q_p(\omega_{u,n})=(\Q_p)_{up,n}$  
is the $n$th Lubin-Tate extension of $\Q_p$ with respect to the
uniformising parameter $up$. 
 
\begin{definition}
We define 
$$[\ \ ]:\Z_p[[X]]\rightarrow W(\Z_p[[X]])$$
to be the unique ring homomorphism such that for $h(X)\in\Z_p[[X]]$,
$[h(X)]$ is the unique Witt vector over $\Z_p[[X]]$ whose ghost vector
is equal to  
$$<h(X),h(f_u(X)),h(f_u(f_u(X))),\ldots>.\enspace$$ 

\end{definition}
See \cite[Lemma 2.1]{Pulita} for a proof of existence and uniqueness.

Following Pulita, for any finite extension $L/\Q_p$ and any $x\in
\bp_L$, we then denote $[h(x)]=[h(X)]|_{X=x}\in W(\bo_L)$.
 
\begin{remark} 
If we take $h(X)=X$ and $x=\omega_{u,n}$ in the above definition, then
$[\omega_{u,n}]\in W(\bo_{K(\omega_{u,n})})$ is the unique Witt vector
with ghost vector 
$$<\omega_{u,n},\omega_{u,n-1},\ldots,\omega_{u,1},0,0,\ldots>\enspace.$$
\end{remark} 

We denote by
$E(X)=\exp\left(X+\frac{X^p}{p}+\frac{X^{p^2}}{p^2}+\cdots\right)$
the Artin-Hasse exponential and recall that $E(X)\in 1+X\Z_p[[X]]$
(see, for example, \cite[I\S9]{Fesenko-Vostokov}). 
\begin{definition}
With $L$ a finite extension of $\Q_p$, let
$\lambda=(\lambda_0,\lambda_1,\ldots)\in W(\bo_L)$. The Artin-Hasse
exponential relative to $\lambda$ is defined as 
$$E(\lambda,X):
=\prod_{i\ge0}E(\lambda_iX^{p^i})
=\exp\left(
\lambda^{(0)}X+\lambda^{(1)}\frac{X^p}{p}+\lambda^{(2)}\frac{X^{p^2}}{p^2}+\ldots
\right)
\enspace.$$ 
\end{definition}
We note that the integrality of the Artin-Hasse exponential and the
fact that $\lambda\in W(\bo_L)$ imply that
$E(\lambda,X)\in1+X\bo_L[[X]]$. Further, we observe that for
$u\in\Z_p^{\times}$  
$$E([\omega_{u,n}]\lambda,X)=\exp
\left(
\omega_{u,n}\lambda^{(0)}X
+\omega_{u,n-1}\lambda^{(1)}\frac{X^p}{p}
+\ldots
+\omega_{u,1}\lambda^{(n-1)}\frac{X^{p^{n-1}}}{p^{n-1}}
\right)
\enspace.$$

We now state two lemmas from \cite{Pulita} that we need to prove our
result. 
\begin{lemma}\label{Pul_key_lemma}
Let $L$ be a finite extension of $\Q_p$. Let $h(X)=\sum_{i\geq 0}a_i
X^i\in \Z_p[[X]]$ and let $[h(x)]=(\lambda_0,\lambda_1,\ldots)\in
W(\bo_L)$ with $x\in\bp_L$ (\textit{i.e.}, $|x|_p<1$). Then, 
$$|a_0|_p=|p|_p^r\text{\ \ \ \ if and only if \ \ \ \ \ }
|\lambda_0|_p,\ldots,|\lambda_{r-1}|_p<1 
\text{\ \ and\ \ } |\lambda_r|_p=1\enspace.$$
Moreover, if $a_0=0$, then $|\lambda_r|_p<1$ for all $r$.
\end{lemma}

\begin{proof} 
\cite[Lemma 2.2]{Pulita}.
\end{proof}

\begin{lemma}\label{Pul_thm2.2}
Let $\lambda=(\lambda_0,\lambda_1,\ldots)\in W(\bo_L)$, then for
$u\in\Z_p^{\times}$, $E([\omega_{u,n}]\lambda,X)$ is over convergent
if and only if $|\lambda_0|_p,\ldots,|\lambda_{n-1}|_p<1$. 
\end{lemma}

\begin{proof} 
This is the equivalence of $2$ and $3$ in \cite[Theorem 2.2]{Pulita}.
\end{proof}

We now state our over-convergence result.
\begin{theorem}\label{thm:power-series}
The power series 
$$
E_{u,n}(X)=\exp
\left(
\sum_{i=0}^{n-1}\frac{\omega_{u,n-i}(X^{p^i}-uX^{p^{i+1}})}{p^i}
\right)$$ 
is over-convergent for all $u\in\Z_p^{\times}$ and $n\in\mathbb{N}$.
\end{theorem}

This is a generalisation of the equivalence of $3$ and $4$ in
\cite[Theorem 2.5]{Pulita} and is essentially proved in the same
way. Letting $u=1$ reduces to Pulita's setup (with a shift in the
numbering). 

\begin{proof}
Let $u\in\Z_p^{\times}$ and $n\in\mathbb N$, then
\begin{eqnarray*}
E_{u,n}(X)
&=&\exp\left(
\left(\sum_{i=0}^{n-1}\frac{\omega_{u,n-i}X^{p^{i}}}{p^i}\right)-
\left(\sum_{i=0}^{n-1}\frac{u\omega_{u,n-i}X^{p^{i+1}}}{p^i}\right)
\right)\\
&=&
\exp(up\omega_{u,n+1}X)
\exp\left(
\sum_{i=0}^{n}\omega_{u,n-i+1}
\left(\frac{\omega_{u,n-i}}{\omega_{u,n-i+1}}-up\right)
\frac{X^{p^i}}{p^i}
\right)\enspace.
\end{eqnarray*}
Letting ${h}(X)=f_u(X)/X-up$ we then have 
$$E_{u,n}(X)
=\exp(up\omega_{u,n+1}X)\, 
E([\omega_{u,n+1}][{h}(\omega_{n+1})],X)
\enspace.$$
We know that $\exp(up\omega_{u,n+1}X)$ is over-convergent, and so we
need only to check the over-convergence of
$E([\omega_{u,n+1}][{h}(\omega_{n+1})],X)$. 

The constant term in ${h}(X)$ is $up-up=0$ and
$|\omega_{u,n+1}|_p<1$. Therefore, if we let
$[{h}(\omega_{n+1})]=(\lambda_0,\lambda_1,\ldots)$, then from Lemma
\ref{Pul_key_lemma} we see that $|\lambda_r|_p<1$ for all
$r$. Over-convergence now follows from Lemma \ref{Pul_thm2.2}.
\end{proof}

%%%%%%%%%%%%%%%%%%%%%%%%%%%%%%%%%%%%%%%%%%%%%%%%%%%%%%%%%%

\subsection{Differential modules with Frobenius structure}

We now explain briefly how these exponentials give a special structure
to certain $p$-adic differential modules. For a detailed account of
differential modules with Frobenius structure, see \cite[Ch. 5 and
  Ch. 17]{Kedlaya}. Recall, specifically, that a \textit{differential
  ring} is a commutative ring $R$ equipped with a derivation
$d:R\rightarrow R$, namely an additive map satisfying
$d(ab)=ad(b)+bd(a)$ for $a,b\in R$. 

\begin{definition}\label{def:diff-mod}
A differential module over a differential ring $(R,d)$ is an
$R$-module $M$ equipped with an additive map $D:M\rightarrow M$
satisfying, for any $a\in R$ and $m\in M$, 
$$D(am)=aD(m)+d(a)m\enspace.$$
Such a $D$ is called a differential operator on $M$ relative to $d$.
\end{definition}
Note that morphisms of $(R,d)$-differential modules are morphisms of
$R$-modules that commute with the differential operators. 

Recall that $L$ is some finite extension of $\Q_p$. 
\begin{definition} 
We define $\mathcal R_L$, the Robba ring with coefficients in $L$, to
be the ring of bidirectional power series over $L$ which converge on
some annulus of outer radius $1$, namely 
$$\mathcal{R}_L=\left\{\sum_{n=-\infty}^{\infty}a_nX^n
\ \left|\
\begin{array}{l}
a_n\in L \text{ for every }n\in\Z,\\ 
 \lim_{+\infty}|a_n|_p\,\alpha^n=0 \text{\ for some $0<\alpha<1$,}\\
\lim_{-\infty}|a_n|_p\,\rho^n=0 \text{\ for all $0<\rho<1$.}
\end{array}\right.
\right\}\enspace.$$
\end{definition}
We equip $\mathcal{R}_L$ with the derivation $\dd_X=X\frac{d}{dX}$. 

Let $M$ be a free rank one $\R_L$-module. In view of Definition
\ref{def:diff-mod}, any differential operator $D$ on $M$ is defined by
giving the image $D(\e)$ of a basis $\e$ of $M$, namely by specifying
$g\in{\mathcal R}_L$ such that $D(\e)=g\e\,$. 
\begin{definition}
Let $h\in\mathcal{R}_L$. A free rank one $\R_L$-differential module
$(N,\Delta)$ is said to be defined by $\dd_X-h$ if there exists a basis
$\varepsilon$ of $N$ such that $\Delta(\varepsilon)=h\varepsilon$.
\end{definition}

\begin{remark}\label{basis_change_remark}
Suppose that $(M,D)$ is defined by $\dd_X-g$, with basis $\e$. If
$f\in\mathcal{R}_L^\times$, $\e'=f\e$ is also a basis of $M$, and
$$D(\e')=fD(\e)+\dd_X(f)\e=\left(g+\frac{\dd_Xf}{f}\right)\e'\enspace,$$
hence $(M,D)$ is also defined by $\dd_X-(g+\frac{\dd_Xf}{f})$, with
basis $\e'$. 
\end{remark}
On the other hand, one easily checks that if $(M,D)$ is also defined by
$\dd_X-h$ for some $h\in{\mathcal R}_L$, then there exists $f\in{\mathcal
R}_L^\times$  such that $h=g+\frac{\dd_Xf}{f}$.
\begin{proposition}
Let $(M,D)$ be defined by $\dd_X-g$ then, for any free rank one
$\R_L$-differential module $(N,\Delta)$, $(N,\Delta)$ is isomorphic to
$(M,D)$ if and only if $(N,\Delta)$ is defined by $\dd_X-g$.
\end{proposition}
\begin{proof}
Suppose $(N,\Delta)$ is defined by $\dd_X-g$, and let $\varepsilon$
(resp. $\e$) be a basis of $N$ (resp. $M$) such that
$\Delta(\varepsilon)=g\varepsilon$
(resp. $D(\e)=g\e$). Let $\varphi$ be the ${\mathcal
  R}_L$-morphism $M\rightarrow N$ defined by
$\varphi(\e)=\varepsilon$, then $\varphi$ is an  isomorphism
of ${\mathcal R}_L$-modules and, for all $a\in{\mathcal R}_L$, 
$$\varphi(D(a\e))
=ag\varepsilon+\dd_X(a)\varepsilon
=a\Delta(\varepsilon)+\dd_X(a)\varepsilon
=D'(\varphi(a\e))
\enspace,$$
namely $\varphi$ is an isomorphism of ${\mathcal R}_L$-differential
modules.

Suppose now that an isomorphism of ${\mathcal R}_L$-differential
modules $\varphi:(M,D)\rightarrow(N,\Delta)$ is given. Let again
$\e$ be a basis of $M$ such that $D(\e)=g\e$ and set
$\varepsilon=\varphi(\e)$, then $\varepsilon$ is a basis of $N$ and  
$$\Delta(\varepsilon)
=\Delta(\varphi(\e))
=\varphi(D(\e))
=g\varepsilon
\enspace,$$
as required.
\end{proof}

We will refer to any $\mathcal R_L$-differential module defined by
$\dd_X-g$ as the $\mathcal R_L$-differential module defined by
$\dd_X-g$. This is a slight abuse of terminology, but is consistent
with the literature. 

We now recall some facts described in \cite[\S1.2.3-4]{Pulita}. Note
that Pulita defines $\mathcal{R}_L$-differential modules in a
different (but equivalent) way. Namely, he defines them as modules over
the non-commutative ring $\mathcal R_L[\dd_X]$  of differential
polynomials with coefficients in $\mathcal R_L$ (and the rule
$\dd_X\,a=a\dd_X+\dd_X(a)$). 

Let $\varphi:L\rightarrow L$ be a continuous $\Q_p$-automorphism
lifting the Frobenius automorphism of $k_L/\F_p$. This extends to a
continuous ring endomorphism  
$$\begin{array}{cccc}
\phi:&\R_L&\longrightarrow&\R_L\\
&\sum a_iX^i&\longmapsto&\sum\varphi(a_i)X^{pi}
\end{array}$$ 
known as an \textit{absolute Frobenius} on $\R_L$. There exists a
functor $\phi^*$ from the isomorphism classes of rank one
$\R_L$-differential modules to themselves, which sends the
$\mathcal{R}_L$-differential module defined by $\dd_X-g(X)$ to the
$\mathcal{R}_L$-differential module defined by  
$$\dd_X-\frac{X}{\phi(X)}\frac{d\phi}{dX}\phi(g(X))
=\dd_X-p\phi(g(X))\enspace.$$

\begin{definition}
A rank one $\mathcal{R}_L$-differential module $M$ has Frobenius
structure (of order $1$) if $M\cong\phi^{*}(M)$. 
\end{definition} 

\begin{theorem}
Let $u\in\mu_{p-1}\subset \Z_p^{\times}$, $f_u(X)=X^p+upX$, and let
$\{\omega_{u,n}\}_{n\ge 0}$ be a coherent set of roots associated to
$f_u(X)$. Then, for all $n\in\mathbb N$, the  
$\R_{\Q_p(\omega_{u,n})}$-differential module defined by 
$$\dd_X-\sum_{i=0}^{n-1}\omega_{u,n-i}X^{p^{i}}$$ 
has Frobenius structure.
\end{theorem}

\begin{proof}  
For $k\in\mathbb N$, we denote by $f_u^k(X)$ the composition of $k$
copies of $f_u(X)$. From standard Lubin-Tate theory we know that
$\Q_p(\omega_{u,n})/\Q_p$ is Galois and of degree $p^{n-1}(p-1)$ and
that $\omega_{u,n}$ has minimum polynomial $f_u^n/f_u^{n-1}$ (this
polynomial has the same degree as the extension and is Eisenstein over
$\Q_p$). 

For our choices of $u$ and $f_u(X)$, we see that $f_u(uX)=uf_u(X)$,
and so ${f_u^n}/{f_u^{n-1}(uX)}={f_u^n}/{f_u^{n-1}(X)}$. This means
that $u\omega_{u,n}$ is also a root of $f_u^n/f_u^{n-1}$, and
$\varphi:\omega_{u,n}\mapsto u\omega_{u,n}$ defines an element of
$\Gal(\Q_p(\omega_{u,n})/\Q_p)$, which is therefore a $\Q_p$-linear 
continuous automorphism of $\Q_p(\omega_{u,n})$. From these
definitions, we also have $\varphi(\omega_{u,m})=u\omega_{u,m}$ for 
all $m\le n$. As $\Q_p(\omega_{u,n})/\Q_p$ is totally ramified, the
residue field extension is trivial and any $\Q_p$-automorphism of
$\Q_p(\omega_{u,n})$ lifts the Frobenius automorphism of
$k_{\Q_p(\omega_{u,n})}/\F_p$.

Let $\phi$ be the absolute Frobenius on
$\mathcal{R}_{\Q_p(\omega_{u,n})}$ coming from $\varphi$ as described
above. Let 
$$g(X)=\sum_{i=0}^{n-1}\omega_{u,n-i}X^{p^i}\enspace$$ 
and note that we have $\phi(g(X))=ug(X^{p})$.

We note that $E_{u,n}(X)=E_{u,n}(-X)^{-1}$, and so the 
over-convergence of $E_{u,n}(X)$, proved in Theorem
\ref{thm:power-series}, implies that $E_{u,n}(X)$ and $E_{u,n}(-X)$
are contained in $\mathcal{R}_{\Q_p(\omega_{u,n})}$; in particular, 
$$E_{u,n}(-X)=\exp\left(
\sum_{i=0}^{n-1}\frac{\omega_{u,n-i}(uX^{p^{i+1}}-X^{p^i})}{p^i}
\right)
\in\mathcal{R}_{\Q_p(\omega_{u,n})}^{\times}\enspace.$$
Let $M$ be the $\mathcal{R}_{\Q_p(\omega_{u,n})}$-differential module
defined by $\dd_X-g(X)$ with basis vector \e. From Remark
\ref{basis_change_remark} we see that $M$ is also defined by
$\dd_X-upg(X^p)$ with basis vector $E_{u,n}(-X)\e$. Finally, we
observe that as $\phi(g(X))=ug(X^{p})$, the differential module
$\phi^*(M)$ is defined by $\dd_X-upg(X^p)$, and so
$M\cong\phi^{*}(M)$. 
\end{proof}

\begin{remark}
A differential module with Frobenius structure is necessarily solvable
in the sense of \cite[Definition 1.7]{Pulita} (see \cite[Remark
  1.5]{Pulita}); in \cite[Theorem 3.1]{Pulita} Pulita completely
describes the isomorphism classes of rank one solvable
$R_{\Q_p(\omega_{u,n})}$-differential modules. Therefore, despite that
fact that the differential modules above are different to those
explicitly described by Pulita, they must each be contained in one of
Pulita's isomorphism classes. 
\end{remark}

\bibliography{bib}
\end{document}